\documentclass[12pt]{article}

\usepackage{amsmath, amssymb, amsthm, graphicx, nag, bbding, microtype, siunitx, setspace, cite, colonequals, graphicx, hyperref, mathrsfs, enumerate}

\newtheorem{thm}{Theorem}[section]
\newtheorem{prop}[thm]{Proposition}

\newtheorem{lemma}[thm]{Lemma}
\newtheorem{fig}[thm]{Figure}

\theoremstyle{definition}
\newtheorem{definition}{Definition}[section]

\title{\bf{On the Long-Repetition-Free 2-Colorability of Trees}}

\author{%
	\makebox[.5\linewidth]{Joseph Antonides}\\Department of Mathematics \\ The Ohio State University \\ Columbus, OH, USA 43210 \\\href{mailto:antonides.4@osu.edu}{antonides.4@osu.edu} \\
	\and \makebox[.5\linewidth]{Claire Kiers}\\Department of Mathematics \\ The University of North Carolina Chapel Hill \\ Chapel Hill, NC, USA 27599 \\\href{mailto:cekiers@live.unc.edu}{cekiers@live.unc.edu} \\
	\and \makebox[.5\linewidth]{Nicole Yamzon\footnote{Corresponding author}}\\Department of Mathematics \\ San Francisco State University \\ San Francisco, CA, USA 94132 \\ \href{mailto:nyamzon@mail.sfsu.edu}{nyamzon@mail.sfsu.edu}\\
}

\date{}

\begin{document}

\maketitle
\vspace{115mm}

\begin{abstract}
A word $\bar{w} = \bar{u}\bar{u}$ is a {\em long square} if $\bar{u}$ is of length at least 3; a word $\bar{w}$ is {\em long-square-free} if $\bar{w}$ contains no sub-word that is a long square. We can use words to generate graph colorings; a graph coloring is called {\em long-repetition-free} if the word formed by the coloring of each path in the graph is long-square-free. Every rooted tree of radius less than or equal to seven is long-repetition-free two-colorable. There exists a class of trees which are not long-repetition-free two-colorable.
\end{abstract}


\section*{Introduction}
The aim of this paper is to introduce the concepts of ``long square'' and ``long repetition,'' and to prove two main results regarding the long-repetition-free 2-colorability of trees. Most terminology relevant to this paper is provided, but the reader is encouraged to reference M. Lothaire \cite{Lothaire} and Bondy and Murty \cite{BondyMurty} for more information regarding combinatorics on words and graph theory, respectively. 

\indent The investigation of words that avoid specific letter patters (such as, for example, squares, cubes, or palindromes) has been a topic of interest in the mathematical sciences and the biological sciences since the early 20th century, beginning largely with the papers of Axel Thue (cf. \cite{Thue1}, \cite{Thue2}). Thue proved that it is impossible to define an infinitely-long, square-free word over a binary alphabet; however, Thue proved that it {\em is} possible to define an infinitely-long, square-free word over a ternary alphabet, and Thue gave a definition for such a word (which today is called the Thue-Morse word). One can use such a word to color the vertices of a graph and, in so doing, create what we call repetition-free 3-colorings. 

\indent We say a word $\bar{w}$ is a ``long square'' if $\bar{w} = \bar{u}\bar{u}$ for some  word $\bar{u}$ where $|\bar{u}| \geq 3$. Much is known regarding words that avoid squares generally (and graph colorings that avoid repetitions generally) - see Grytczuk \cite{Grytczuk} for a survey of such results; however, little is known about words that allow for squares of lengths 2 and 4. 

\indent In this article, we use two results regarding palindromes (in particular, ``long palindromes,'' or palindromes of length at least 4) to prove two main results regarding the long-repetition-free 2-colorability of trees. The first main result asserts that every rooted tree of radius less than or equal to seven is long-repetition-free 2-colorable. The second main results answers the question of whether or not every tree is long-repetition-free 2-colorable. We construct a tree called a Tyler tree of height 8 and prove that it is not long-repetition-free 2-colorable.

\section{Terminology and Notation}
We begin by providing several useful definitions and detailing the notation that is used throughout this paper. 

\begin{definition}
	An {\em alphabet} is a finite set of elements called {\em letters}, and we typically denote alphabets by a single capital letter (for example, $A$). The set of all words over alphabet $A$ is denoted $A^*$. We typically denote words of arbitrary length by a single lower-case letter with an overhead bar (for example, $\bar{w}$), and we typically denote letters (words of length 1) by a single lower-case letter with an overhead dot (for example, $\dot{a}$). We denote by $\epsilon$ the word of length $0$, called the {\em empty word}. We denote the concatenation of words $\bar{w}$ and $\bar{v}$ by $\bar{w}\bar{v}$ or, if we wish to stress the significance of concatenation, by $\bar{w} \cdot \bar{v}$. If a word $\bar{u}$ contains sub-word $\bar{w}$, we say $\bar{w} \leq \bar{u}$. 
\end{definition}

\begin{definition}
	For word $\bar{w} = \dot{a}_1 \dots \dot{a}_n$ over alphabet $A$, we define the {\em reversal} of $\bar{w}$ to be $\overleftarrow{w} = \dot{a}_n \dots \dot{a}_1$. A {\em palindrome} is a word $\bar{w}$ such that $\bar{w} = \overleftarrow{w}$. We define a {\em long palindrome} to be a palindrome of length at least 4. Word $\bar{w}$ is a square if there exists some non-$\epsilon$ word $\bar{u}$ such that $\bar{w} = \bar{u}\bar{u}$. Word $\bar{w}$ is {\em square-free} if $\bar{w}$ contains no sub-word that is a square. Similarly, word $\bar{w} = \bar{u}\bar{u}$ is a {\em long square} if $|\bar{u}| \geq 3$. Word $\bar{w}$ is {\em long-square-free} if $\bar{w}$ contains no sub-word that is a long square. 
\end{definition}

\begin{definition}
	Let $G = (V,E)$ be a simple, connected graph, and let $A = \{1, \dots, k\} \subset \mathbb{N}$. Graph $G$ is said to be {\em long-repetition-free $k$-colorable} if there exists a coloring $\chi : V \rightarrow A$ such that, for every path $P = u_1 e_1 \dots e_{n-1} u_n \subseteq G$, the word $\chi(u_1)\dots \chi(u_n)$ is long-square-free. 
\end{definition}

\begin{definition}
	Let $G=(V,E)$ be a connected graph. The {\em distance} between vertices $u,v \in V$ is defined to be the number of edges contained in the shortest path from $u$ to $v$. The {\em eccentricity} of vertex $v$ is the maximum distance between $v$ and any other vertex in $G$. The {\em radius} of graph $G$ is the minimum eccentricity of its vertices, and the {\em center} of graph $G$ is the sub-graph induced by the vertices of minimum eccentricity. 
\end{definition}

\begin{definition}
	A tree $T = (V,E)$ is called a {\em rooted tree} if one vertex has been distinguished to be called the {\em root} of $T$. The edges of a rooted tree are described with an implicit orientation away from the root vertex. In a rooted tree, the {\em parent} of a vertex $v$ is the vertex adjacent to $v$ on the path to the root vertex. A {\em child} (plural, {\em children}) of vertex $v$ is a vertex of which $v$ is the parent. A {\em descendant} of a vertex $v$ is any vertex which is either a child of $v$ or is, recursively, the descendant of any of the children of $v$. A {\em sibling} to a vertex $v$ is any other vertex in $V$ which has the same parent as $v$. The {\em $k$\textsuperscript{th} generation} of root $r$ of tree $T$, denoted $G_k(T)$, is the set of descendant vertices of distance $k$ from $r$. The {\em height} of $T$ is the number of generations in $T$ from $r$, with the convention that $G_0(T) = \{r\}$. 
\end{definition}

In particular, it is of note that if $T = (V,E)$ is a rooted tree, then for root $r \in T$ and vertex $v \in V$ there exists a unique paths from $r$ to $v$ and from $v$ to $r$. As a matter of notation, it is perhaps noteworthy that while the $k\textsuperscript{th}$ generation from $r$ is of course dependent on our choice of $r$, we simply notate this $G_k(T)$ since, for this paper, the particular vertex distinguished as the root is not particularly relevant and furthermore does not change once the root has been chosen. 

\begin{definition}
	The {\em Tyler tree of height $n$}, denoted $\mathbb{T}_n = (\mathbb{V}_n,\mathbb{E}_n)$, is the rooted tree with root $r$ in which each vertex in the $j$\textsuperscript{th} generation from $r$ has $2^{n-j}+1$ children, for every $0 \leq j \leq n-1$. Computationally, $|\mathbb{V}_n| = 1 + \displaystyle\sum_{j=0}^{n-1}\displaystyle\prod_{k=n-j}^{n} (2^{k}+1)$. A {\em binary tree} of height $n$ has at most $2^{n+1} - 1$ vertices. The reader can verify that the Tyler tree of height $n$ contains $2^n + 1$ Tyler trees of height $n-1$. 
\end{definition}

\section{A Class of Trees that is Long-Repetition-Free 2-Colorable}

The goal of the present section is to prove the first of our two main results: Proposition \ref{SmallTrees}, which asserts that every rooted tree of radius less than or equal to $7$ can be long-repetition-free 2-colored. 

In order to prove Proposition \ref{SmallTrees}, we rely on Lemma \ref{abxba} which we now state and prove. 

\begin{lemma}\label{abxba}
	Let $A = \{\dot{a},\dot{b}\}$, and let $\bar{w} \in A^*$. If $\bar{w} = \dot{a}\dot{b}\bar{x}\dot{b}\dot{a}$ for some $\bar{x} \in A^*$, then $\bar{w}$ contains a long palindrome. 
\end{lemma}

\begin{proof}
	We consider two cases: case (1), if $\bar{x}$ contains no square of length $2$, and case (2), if $\bar{x}$ contains at least one square of length $2$. \\
	
	Suppose first that $\bar{x}$ contains no square of length $2$. If $|\bar{x}| \leq 1$, then clearly we have the desired long palindrome. If $|\bar{x}| > 2$, then since by assumption $\bar{x}$ contains no square of length $2$, it must be the case that $\bar{x}$ is an alternating sequence of $\dot{a}$'s and $\dot{b}$'s. It follows that $\bar{x}$ must contain a long palindrome of the form $\dot{a}\dot{b}\dot{b}\dot{a}$ or $\dot{a}\dot{b}\dot{a}\dot{b}\dot{a}$. 
		
	Suppose next that $\dot{y}\dot{y} \leq \bar{x}$ for some $\dot{y} \in A$. If follows that, for some $\dot{l}_1 , \dot{l}_2, \dot{r}_1, \dot{r}_2 \in A$, we have $\dot{l}_1\dot{l}_2\dot{y}\dot{y}\dot{r}_1\dot{r}_2 \leq \bar{w}$. If $\dot{l}_2 = \dot{r}_1$, then we have the desired long palindrome. Suppose $\dot{l}_2 \neq \dot{r}_1$; then either $\dot{l}_2 = \dot{y} \neq \dot{r}_1$ or $\dot{l}_2 \neq \dot{y} = \dot{r}_1$. 
	
	Suppose first that $\dot{l}_2 = \dot{y} \neq \dot{r}_1$, from which follows $\dot{l}_1 \dot{l}_2 \dot{y}\dot{y}\dot{r}_1\dot{r}_2 = \dot{l}_1 \dot{l}_1 \dot{y}\dot{y}\dot{r}_1 \dot{r}_2 \leq \bar{w}$. If $\dot{l}_1 = \dot{r}_1$, then we have the desired long palindrome; otherwise, we have $\dot{y}\dot{y}\dot{y}\dot{y}\dot{r}_1\dot{r}_2 \leq \bar{w}$ since $\dot{l}_2 = \dot{y} \neq \dot{r}_1$ by assumption, which contains the desired long palindrome. 
	
	Now suppose $\dot{l}_2 \neq \dot{y} = \dot{r}_1$, from which follows $\dot{l}_1 \dot{l}_2 \dot{y}\dot{y}\dot{r}_1\dot{r}_2 = \dot{l}_1 \dot{l}_2 \dot{y}\dot{y}\dot{y}\dot{r}_2 \leq \bar{w}$. Either $\dot{l}_1 = \dot{y}$ or $\dot{l}_1 = \dot{l}_2$. In the former case, then either we have $\dot{r}_2 = \dot{y}$, in which case we have $\dot{l}_1 \dot{l}_2 \dot{y}\dot{y}\dot{r}_1\dot{r}_2 = \dot{y}\dot{l}_2 \dot{y}\dot{y}\dot{y}\dot{y}$, or we have $\dot{r}_2 = \dot{l}_2$, in which case we have $\dot{l}_1 \dot{l}_2 \dot{y}\dot{y}\dot{r}_1\dot{r}_2 = \dot{y}\dot{l}_2\dot{y}\dot{y}\dot{y}\dot{l}_2$; in either case, we have the desired long palindrome. In the latter case where $\dot{l}_1 = \dot{l}_2$, then either we have $\dot{r}_2 = \dot{l}_2$ in which case we have $\dot{l}_1 \dot{l}_2 \dot{y}\dot{y}\dot{r}_1\dot{r}_2 = \dot{l}_2 \dot{l}_2 \dot{y}\dot{y}\dot{y}\dot{l}_2$, or we have $\dot{r}_2 = \dot{y}$, in which case we have $\dot{l}_1 \dot{l}_2 \dot{y}\dot{y}\dot{r}_1\dot{r}_2 = \dot{l}_2 \dot{l}_2 \dot{y} \dot{y} \dot{y} \dot{y}$; in either case, we have the desired long palindrome. 
\end{proof}

With Lemma \ref{abxba} in hand, we move to state and prove the first main result of this paper. 

\begin{prop}\label{SmallTrees}
	Every rooted tree of radius less than or equal to $7$ is long-repetition-free 2-colorable.
\end{prop}

\begin{proof}
	Let $T = (V,E)$ be a rooted tree of radius less than or equal to $7$. Let $v \in V$ be a vertex in the center of $T$, and designate $v$ as the root. Let $j$ be the number of generations of $v$, and note that $j \leq 7$. Let $\bar{a} = \dot{a}_1 \dots \dot{a}_7 = 00010111$, and let $\phi: V \rightarrow \{0,1\}$ be the coloring defined by $\phi(u) = \dot{a}_i$, for every vertex $u$ in the $i\textsuperscript{th}$ generation from $v$ and for every $1 \leq i \leq j$. We prove that $\phi$ is a long-repetition-free 2-coloring.
	
	By way of contradiction, suppose there exists a path $P = p_1 \dots p_{2k} \in T$ for $k \geq 3$ such that, for words $\bar{u}' = \phi(p_1) \dots \phi(p_k)$ and $\bar{u}'' = \phi(p_{k+1}) \dots \phi(p_{2k})$, we have $\bar{u}' = \bar{u}''$. Since $P$ contains a long repetition and since $\bar{a} = 00010111$ is long-square-free, $P$ cannot be a path of vertices each from different generations. So, there exists some vertex $p_j \in P$ such that $p_{j-1}$ and $p_{j+1}$ are in the same generation; however, it must be the case that $p_j = v$, for in no other way can adjacent vertices of $p_j$ be in the same generation. It follows that $p_{j-1}$ and $p_{j+1}$ are in the first generation from $v$, and the word $\bar{u}'\bar{u}''$ is of the form $\bar{u}'\bar{u}''= \dot{a}_{j-1} \dots \dot{a}_1 \dot{a}_0 \dot{a}_1 \dots \dot{a}_{j-1} \dot{a}_{j} \dots \dot{a}_{2k-j}$. Assume that $j \leq k$ (for the argument that follows is similar for $j > k$), and we consider the cases where $j<k$ and $j = k$.
	
	Case 1: suppose $j < k$. Then $\dot{a}_1 \dot{a}_0 \dot{a}_1 \leq \bar{u}'$. Since $\bar{u}' = \bar{u}''$, we have $\dot{a}_1 \dot{a}_0 \dot{a}_1 \leq \bar{u}''$ as well. Thus, since $j < k$, we have a word of the form $\dot{a}_0 \dot{a}_1 \bar{x} \dot{a}_1 \dot{a}_0 \leq \bar{a}$, where $\bar{x} \in \{0,1\}^*$.  By Lemma \ref{abxba}, we find that $\bar{a}$ contains a palindrome of length 4 or 5. However, $\bar{a} = 00010111$ has no palindrome of length greater than 3, which is a contradiction. 
	
	Case 2: suppose $j = k$. Then either $\bar{a} = \bar{u}'$ and $\bar{a} = \overleftarrow{u}''$, or $\bar{a} = \overleftarrow{u}'$ and $\bar{a} = \bar{u}''$. Assuming, without loss of generality, that the latter is true, it follows that $\bar{u}''$ (and therefore $\bar{a}$) is of the form $\bar{u}' = \dot{a}_0 \dot{a}_1 \dots \dot{a}_1 \dot{a}_0$. Thus it follows by Lemma \ref{abxba} that $\bar{a} = 00010111$ contains a palindrome of length 4 or of length 5, which is a contradiction. Hence, $\phi$ is a long-repetition-free 2-coloring.
\end{proof}

\section{A Class of Trees that is Not Long-Repetition-Free 2-Colorable}
One may ask whether or not every tree is long-repetition-free 2-colorable. In this section, we show that a certain class of trees require a minimum of 3 colors for a long-repetition-free coloring, making 3 colors the lowest bound for long-repetition-free colorings of trees. In order to gain this result, we first prove Lemma \ref{LPalindrome} and Lemma \ref{BinaryTree}. 

\begin{lemma}\label{LPalindrome} Every binary word of length at least $9$ contains a long palindrome. \end{lemma}

\begin{proof} It is clear that we need only consider binary words of length $9$, for any binary words of length greater than $9$ contains a binary word of length $9$. Let $\dot{a}_{1}\dot{a}_{2} \dots \dot{a}_{9}$ be a word over alphabet $A=\{0,1\}$. Define $\dot{a}^{c}_{i}$ to be the binary complement of $\dot{a}_i$; more precisely, $\dot{a}^c_i \equiv \dot{a}_i+1 \mod{1}$ for every $i \leq 9$. The reader can verify that for $\dot{a}_1 \dots \dot{a}_9$ to be long-palindrome-free, the following conditions must hold:
	\begin{enumerate}[(i)]
		\item If $\dot{a}_i = \dot{a}_{i+3}$, then $\dot{a}_{i+1} = \dot{a}^{c}_{i+2}$, for $1 \leq i \leq 6$; 
		\item If $\dot{a}_i = \dot{a}_{i+4}$, then $\dot{a}_{i+1} = \dot{a}^{c}_{i+3}$, for $1 \leq i \leq 5$; 
		\item If $\dot{a}_i = \dot{a}_{i+1}$, then $\dot{a}_{i+1} = \dot{a}^{c}_{i+2}$, for $2 \leq i \leq 6$. 
	\end{enumerate}
	
	\noindent
	Suppose without loss of generality that $\dot{a}_0 = 0$. We consider four cases. First, suppose $\dot{a}_0 = \dot{a}_4 = \dot{a}_5 = 0$. By (i), it follows that $\dot{a}_2 = \dot{a}_3^{c}$. Then $\dot{a}_2 = 1$ and $\dot{a}_3 = 0$, for otherwise we have long palindrome $00100$. However, we now have no choice for $\dot{a}_6$, for $\dot{a}_6 = 0$ yields long palindrome $0000$, and $\dot{a}_6 = 1$ yields long palindrome $10001$. 
	
	\indent Suppose next that $\dot{a}_1 = \dot{a}_4 = \dot{a}_5^c = 0$. By (i), it follows that $\dot{a}_2 = \dot{a}_3^c$. Then $\dot{a}_2 = 0$ and $\dot{a}_3 = 1$, for otherwise we have long palindrome $1001$. Then $\dot{a}_6 = 1$, for otherwise we have long palindrome $01010$, and so $\dot{a}_7 = 0$, for otherwise we have long palindrome $0110$. However, we now have no choice for $\dot{a}_8$, for $\dot{a}_8 = 0$ yields long palindrome $01110$, and $\dot{a}_8 = 1$ yields long palindrome $1111$. 
	
	\indent Suppose that $\dot{a}_1 = \dot{a}_4^c = \dot{a}_5^c = 0$. By (iii), $\dot{a}_6 = 0$, and it follows that $\dot{a}_3 = 1$ for otherwise we have long palindrome $0110$. However, we now have no choice for $\dot{a}_2$, for $\dot{a}_2 = 0$ yields long palindrome $01110$, and $\dot{a}_2 = 1$ yields long palindrome $1111$. 
	
	\indent Finally, suppose that $\dot{a}_1 = \dot{a}_4^c = \dot{a}_5 = 0$. By (ii), $\dot{a}_2 = 1$, and it follows that $\dot{a}_3 = 0$ for otherwise we have long palindrome $0110$. Then $\dot{a}_6=1$, for otherwise we have long palindrome $00100$; $\dot{a}_7 = 1$, for otherwise we have long palindrome $01010$; and $\dot{a}_8 = 1$, for otherwise we have $0110$. However, we now have no choice for $\dot{a}_9$, for $\dot{a}_9 = 0$ yields long palindrome $01110$, and $\dot{a}_9 =1 $ yields long palindrome $1111$.  \end{proof}
	
\begin{lemma}\label{BinaryTree} Let $\mathbb{T}_n$ denote the Tyler tree of height $n$ with root $r$, and apply an arbitrary 2-coloring $\phi$ to $\mathbb{T}_n$. Then $\mathbb{T}_n$ contains a binary sub-tree of height $n$, denoted $\mathbb{B}_n$, with root $r$ such that, for every $u,v \in G_i(\mathbb{B}_n)$, $\phi(u) = \phi(v)$, for all $0 \leq i \leq n-1$.  \end{lemma}

\begin{proof} We prove the result by induction on height $n$. \\
	
	Tyler tree $\mathbb{T}_1$ contains exactly four vertices: root $r$ and 3 children vertices adjacent to $r$. By the pigeonhole principle, at least 2 of these 3 children vertices must be colored with the same color. Thus, $\mathbb{T}_1$ contains a binary tree of height $1$ whose root is $r$ and where all the vertices on the same level have the same color. See Figure \ref{fig:basecase} for an example of $\mathbb{T}_1$ and $\mathbb{T}_2$ each with an embedded binary tree colored by generations.  \\

	\begin{center}{\includegraphics[scale=1.00]{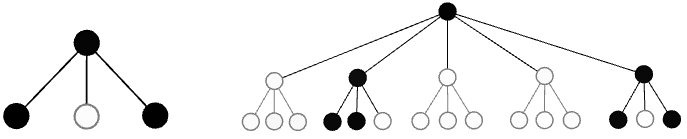}}
	\begin{fig}\label{fig:basecase} The Tyler tree of height $1$ (left) and the Tyler tree of height $2$ (right) each contain a binary sub-tree in which every vertex in each generation from the root has the same color. \end{fig} \end{center}

	Suppose every Tyler tree of height $k \leq n-1$ contains a binary sub-tree such that, for every $u,v \in G_i(\mathbb{B}_k)$, we have $\phi(u) = \phi(v)$, for every $0 \leq i \leq k-1$. Let $\mathbb{T}_{k+1}$ denote the Tyler tree of height $k+1$ with root $r$, and apply arbitrary 2-coloring $\phi$ to $\mathbb{T}_{k+1}$. By definition, $r$ has $2^{k+1} + 1$ children, and each of these children are the root of a Tyler tree of height $k$. Therefore, $\mathbb{T}_{k+1}$ contains $2^{k+1}+1$ Tyler trees of height $k$ as sub-trees, and, by the inductive hypothesis, it follows that each of these Tyler trees of height $k$ contain a binary sub-tree of height $k$ such that, for every $u,v \in G_i(\mathbb{B}_k)$, we have $\phi(u) = \phi(v)$, for every $0 \leq i \leq k-1$. 
	
	Each of these $2^{k+1}+1$ binary sub-trees has $k$ generations, and, given the coloring constraints provided in the inductive hypothesis, there are exactly $2^k$ possible 2-colorings of these binary sub-trees. Attach to these binary sub-trees the root $r$. As there are $2$ possible 2-colorings of $r$, there are a total of $2^{k+1}$ possible 2-colorings of this binary sub-tree of height $k+1$. Since there are $2^{k+1} + 1$ binary sub-trees of height $k$ and only $2^{k+1}$ possible 2-colorings, it follows that at least two binary sub-trees of height $k$ must be colored identically. The sub-tree induced by these two identically-colored binary trees of height $k$, when attached to $r$, form a binary tree of height $k+1$ in which, for every $u,v \in G_i(\mathbb{B}_{k+1})$, we have $\phi(u) = \phi(v)$, for all $0 \leq i \leq k$, thereby completing the proof. \end{proof}

\begin{prop}\label{Thm:TylerTree} If $n \geq 8$, then $\mathbb{T}_{n}$ is not long-repetition-free 2-colorable. \end{prop}

\begin{proof} Let $\mathbb{T}_8$ denote the Tyler tree of height $8$ with root $r$. We need only consider $\mathbb{T}_8$ to prove the result. Apply an abritrary 2-coloring $\phi$ to $\mathbb{T}_8$. By Lemma \ref{BinaryTree}, $\mathbb{T}_8$ contains a binary sub-tree $\mathbb{B}_8$ in which, for each $u,v \in G_i (\mathbb{B}_8)$, we have $\phi(u) = \phi(v)$, for all $0 \leq i \leq 7$.\\
	
	Pick vertices $u_0, \dots, u_7$ such that $u_i \in G_i (\mathbb{B}_8)$ for each $0 \leq i \leq 7$, and such that $u_0 , \dots , u_7$ is a simple path. Let $\dot{b}_i$ denote the letter associated with $\phi(u_i)$, and let $\bar{b} = \dot{b}_0 \dots \dot{b}_7$. By Lemma \ref{LPalindrome}, $\bar{b}$ contains a palindrome $\bar{p}$ of length 4 or of length 5. Suppose $\bar{p} = \dot{b}_j \dots \dot{b}_{j+m}$, where $0 \leq j \leq 4$ and where $m = 3$ or $m = 4$. Consider the path $P = u_{j+m} \dots u_{j+1} u_j u_{j+1}' \dots u_{j+m}'$, where $u_i' \in G_i (\mathbb{B}_8)$ for each $0 \leq i \leq j + m$.  
	
	Let $\bar{c}$ be the word associated with $\phi(P) = \phi( u_{j+m} \dots u_{j+1} u_j u_{j+1}' \dots u_{j+m}')$. If $|\bar{p}| = 4$, then either $\bar{c} = \dot{c}_1\dot{c}_1\dot{c}_1\dot{c}_1\dot{c}_1\dot{c}_1\dot{c}_1$, or $\bar{c} = \dot{c}_1\dot{c}_2\dot{c}_2\dot{c}_1\dot{c}_2\dot{c}_2\dot{c}_1$, where $c_i \in \{0,1\}$. In either case, $\bar{c}$ contains a long square. If $|\bar{p}| = 5$, then there are four possibilities: (1) $\bar{c} = \dot{c}_1\dot{c}_1\dot{c}_1\dot{c}_1\dot{c}_1\dot{c}_1\dot{c}_1\dot{c}_1\dot{c}_1$, (2) $\bar{c} = \dot{c}_1\dot{c}_2\dot{c}_1\dot{c}_2\dot{c}_1\dot{c}_2\dot{c}_1\dot{c}_2\dot{c}_1$, (3) $\bar{c} = \dot{c}_1\dot{c}_1\dot{c}_2\dot{c}_1\dot{c}_1\dot{c}_1\dot{c}_2\dot{c}_1\dot{c}_1$ or (4) $\bar{c} = \dot{c}_1\dot{c}_2\dot{c}_2\dot{c}_2\dot{c}_1\dot{c}_2\dot{c}_2\dot{c}_2\dot{c}_1$. In any of these four cases, $\bar{c}$ contains a long square. Since $\bar{c}$ contains a long square regardless of the length of $\bar{p}$, it follows that $P$ contains a long repetition, and thus $\mathbb{T}_8$ is not long-repetition-free 2-colorable. 
\end{proof}

\raggedright
\section*{Acknowledgements}
This research was conducted in part under the direction of David Milan at the 2013 Research Experience for Undergraduates program at the University of Texas at Tyler and was supported in part by the National Science Foundation (Grant DMS-1062740). This research was also conducted in part under the direction of Annika Miller at Susquehanna University.

\bibliography{ReferencesPaper1justTrees2ndEd}{}
\bibliographystyle{plain}

\end{document}